\documentclass[12pt]{amsart}
\usepackage{amsmath,amsthm,amssymb,amsfonts,dsfont}
\usepackage{graphicx}
\usepackage{color}
\usepackage[utf8]{inputenc}

\newtheorem{theorem}{Theorem}[section]
\newtheorem{lemma}[theorem]{Lemma}

\newtheorem{proposition}[theorem]{Proposition}

\newtheorem{corollary}[theorem]{Corollary}
\theoremstyle{definition}
\newtheorem{definition}[theorem]{Definition}

\numberwithin{equation}{section}

\def\llll{\longrightarrow}

\def\supp{{\rm supp \;}}

\def\sep{{ \ \  }}
\def\sem{{\ \ \ \  }}
\def\seg{{\ \ \ \  \ \  }}

\title[Bishop-Phelps-Bollob{\'a}s  property] {Bishop-Phelps-Bollob{\'a}s  property  for positive \\ operators  when the domain is $C_0(L) $}

\author[M.D. Acosta]{Mar\'{\i}a D. Acosta}
\address{University of Granada, Faculty of Sciences,
	Department of Mathematical Analysis and Institute of Mathematics (IEMath-GR), 18071 Granada, Spain}
\email{dacosta@ugr.es}
\author[M. Soleimani]{Maryam Soleimani-Mourchehkhorti}
\address{School of Mathematics, Institute for Research in Fundamental Sciences (IPM), P.O. Box: 19395-5746, Tehran, Iran}
\email{m-soleimani85@ipm.ir}

\thanks{The  first  author was  supported  by Junta de Andaluc\'{\i}a grant  FQM--185,  by Spanish MINECO/FEDER grant PGC2018-093794-B-I00 	
	and also by     Junta de Andaluc\'{\i}a grant  A-FQM-484-UGR18. The second author was   supported by a grant from IPM}

\keywords{Banach space,  operator, Bishop-Phelps-Bollob{\'a}s  theorem, Bishop-Phelps-Bollob{\'a}s  \linebreak[4] pro\-per\-ty for positive operators.}

\begin{document}

\subjclass[2010]{Primary 46B20; Secondary 46B42.}

\begin{abstract}
	Recently it was introduced the  so-called  Bishop-Phelps-Bollob{\'a}s property  for positive operators between Banach lattices.
	In this paper we prove that the pair
	$(C_0(L), Y) $ has the Bishop-Phelps--Bollob{\'a}s property for positive operators, for any locally compact Hausdorff topological space $L$, whenever  $Y$ is a uniformly monotone Banach lattice with a weak unit.
	In case that the space $C_0(L)$ is separable, the same statement holds for any uniformly monotone Banach lattice $Y .$   We also show the following partial converse of the main result. In case that $Y$ is a strictly monotone Banach lattice,  $L$ is a locally compact Hausdorff topological space that contains at least two elements and the pair $(C_0(L), Y )$ has the  Bishop-Phelps--Bollob{\'a}s property for positive operators then $Y$ is uniformly monotone.
	\end{abstract}

\maketitle

\section{Introduction}

	Bishop--Phelps theorem \cite{BP} states that every continuous linear functional on a Banach space can be approximated (in norm) by norm attaining functionals. Bollob\'{a}s proved a “quantitative version” of that result \cite{Bol}. In order to state such result, we denote by $B_X ,$  $S_X$ and $X^*$ the closed unit ball, the unit sphere and the topological dual of a Banach space $X$, respectively. If $X$ and $Y$ are both real or both complex Banach spaces, $L(X,Y )$ denotes the space of (bounded linear) operators from $X$ to  $Y,$  endowed with its usual operator norm.

\vspace{5mm}

\textit{	Bishop--Phelps--Bollob\'{a}s theorem } (see \cite[Theorem 16.1]{BoDu}, or \cite[ Corollary 2.4]{CKMMR}) Let $X$ be a Banach space and $0 <\varepsilon < 1 .$ Given $x \in B_X$ and $x^* \in B_{X^*}$ with $\vert 1 - x^* (x) \vert < \frac{\varepsilon ^2}{4} ,$ there are elements  $y \in B_X$ and $y^* \in B_{X^*}$	such that $y^*(y) =1 ,$ $\Vert y - x \Vert < \varepsilon$ and $\Vert y^* - x^* \Vert < \varepsilon .$

\vspace{5mm}

In 2008 the study of extensions of the Bishop--Phelps--Bollob\'{a}s theorem for operators was initiated by Acosta, Aron, García and Maestre \cite{AAGM}.   The survey \cite{Acbc} contains many results on that topic.

\begin{definition}
	\label{def-BPBP}
	(\cite[ Definition 1.1]{AAGM}). Let $X$ and $Y$ be either real or complex
	Banach spaces. The pair $(X, Y )$ is said to have the Bishop--Phelps--Bollob\'{a}s property for operators if for every $0 <\varepsilon < 1$ there exists $ 0 <\eta (\varepsilon) <\varepsilon $ such that for every $S \in S_{L(X,Y )},$ if $x_0 \in S_X$ satisfies $\Vert S(x_0)\Vert > 1 - \eta (\varepsilon) ,$  then there exist an element $u_0 \in S_X$ and an operator  $T \in S_{L(X,Y )}$  satisfying the following conditions
	$$
	\Vert T(u_0) \Vert =1, ~~~~ \Vert u_0 -x_0\Vert < \varepsilon ~~~~ , \Vert T-S \Vert < \varepsilon .
	$$
	\end{definition}

Recently in \cite{ASposVM}, the authors  introduced a version of  Bishop--Phelps--Bollob\'{a}s property for positive operators between two Banach lattices. Let us mention that the only difference between this property and the previous one is that in the new property the operators appearing in Definition \ref{def-BPBP} are positive. In the same paper it is shown that the pairs $(c_0,L_1(\nu))$ and $(L_\infty (\mu),L_1(\nu))$ have the  Bishop-Phelps-Bollob\'{a}s property for positive operators for any positive measures $\mu$  and $\nu$ (see \cite[Theorems 1.7 and 1.6]{ASposVM}).
The paper \cite{ASposBull}  contains some extensions of those results. More precisely, it is proved that 
the pair $(c_0,Y )$ has the Bishop-Phelps-Bollob\'{a}s property for positive operators whenever $Y$ is a uniformly monotone Banach lattice (see \cite[Corollary 3.3]{ASposBull}). It is  also shown  that the pair $(L_\infty (\mu),Y)$ has the Bishop-Phelps-Bollob\'{a}s property for positive operators for any positive measure $\mu$ if $Y$ is a uniformly monotone Banach lattice with a weak unit (see \cite[Corollary 2.6]{ASposBull} ).

The goal of this paper is to obtain a far reaching extension of those results. To be precise, we prove that the pair
$(C_0(L), Y) $ has the Bishop-Phelps--Bollob{\'a}s property for positive operators, for any locally compact Hausdorff topological space $L$, whenever  $Y$ is a uniformly monotone Banach lattice with a weak unit. If  $C_0(L)$ is separable, the same statement holds for any uniformly monotone Banach lattice $Y .$ Further we show that these results are optimal in case that $Y$ is strictly monotone.   That is, if  $Y$ is a strictly monotone Banach lattice and $L$ is a locally compact Hausdorff topological space that contains at least two elements, if the pair $(C_0(L), Y )$ has the  Bishop-Phelps-Bollob{\'a}s property for positive operators then $Y$ is uniformly monotone.

\section{The results}

\begin{definition}
	 [{\cite[Definition 1.3]{ASposVM}}]
	\label{def-BPBp-pos}
	Let $X$  and $Y$ be    Banach lattices. The pair $(X,Y )$ is said to have the {\it Bishop-Phelps-Bollob{\'a}s property for  positive operators}   if for every $  0 < \varepsilon  < 1 $  there exists $ 0< \eta (\varepsilon) < \varepsilon $ such that for every $S\in S_{L(X,Y)}$, such that $S \ge 0$,  if $x_0 \in S_X$ satisfies
	$ \Vert S (x_0) \Vert > 1 - \eta (\varepsilon)$, then
	there exist an element $u_0 \in S_X$  and a positive  operator $T \in S_{L(X,Y )}$ satisfying the following conditions
	$$
	\Vert T (u_0) \Vert =1, \sem \Vert u_0- x_0 \Vert < \varepsilon \seg \text{and}
	\sem \Vert T-S \Vert < \varepsilon.
	$$
\end{definition}

	In order to refine the result, we recall  the following  version of \cite[Definition 1.3]{ABG-jmaa-2014},  which was introduced in \cite{ASposVM}.

\begin{definition} 
	[{\cite[Definition 1.3]{ASposVM}}]  
	\label{def-BPBp-pos-subs}
	Let $X$  and $Y$ be    Banach lattices and $M $ a subspace of $L(X,Y).$  The subspace $M$  is said to have the {\it Bishop-Phelps-Bollob{\'a}s property for  positive operators}   if for every $  0 < \varepsilon  < 1 $  there exists $ 0< \eta (\varepsilon) < \varepsilon $ such that for every $S\in S_M$, such that $S \ge 0$,  if $x_0 \in S_X$ satisfies
	$ \Vert S (x_0) \Vert > 1 - \eta (\varepsilon)$, then
	there exist an element $u_0 \in S_X$  and a positive  operator $T \in S_M$ satisfying the following conditions
	$$
	\Vert T (u_0) \Vert =1, \sem \Vert u_0- x_0 \Vert < \varepsilon \seg \text{and}
	\sem \Vert T-S \Vert < \varepsilon.
	$$
\end{definition}

We will use the notions of  strictly monotone and  uniformly monotone Banach lattice, that we recall now.

\begin{definition}
	\label{def-SM-UM}
	Let $X$ be a real Banach lattice.   $X$ is \textit{strictly monotone} if for any $x,y \in X^+$ such that $x \le y$ and $x \ne y,$ it is satisfied that $\Vert x \Vert < \Vert y \Vert.$  The Banach 
lattice $X$ is \textit{uniformly monotone},
	if for every $0<\varepsilon<1,$ there is 
	$0  <  \delta \leq \varepsilon$ satisfying the following property
	$$
	x , y  \in X^+ ,\Vert x \Vert =1   , \sep  \Vert x+ y\Vert  \leq 1 + \delta  \Rightarrow\ \Vert y\Vert \leq \varepsilon .
	$$
\end{definition} 

The following characterization of  uniform monotonicity can be found in \cite[Proposition 4.2]{ASposBull}.

\begin{proposition}
	\label{pro-char-UM}
	Let $Y$ be a Banach lattice. The following conditions are equivalent.
	
	\begin{enumerate}
		\item[1)]  $Y$ is uniformly monotone.
		
		\item[2)]  For  every $0< \varepsilon <1$, there is $ \eta (\varepsilon) > 0$ satisfying 
		$$
		u \in Y, v \in S_Y,  \sem 0 \le u\le v \sem \text{and} \sem  \|v - u\|  >   1 - \eta (\varepsilon) \sep \Rightarrow \sep \| u\| \leq  \varepsilon .
		$$
		
		\item[3)]  For  every $0< \varepsilon <1$, there is $ \eta (\varepsilon) > 0$ satisfying 
		$$
		u, v \in Y,  \sem 0 \le u\le v \sem \text{and} \sem  \|v - u\|  >  ( 1 - \eta (\varepsilon) ) \|v\| \sep \Rightarrow \sep \| u\| \leq  \varepsilon  \|v\|.
		$$
		
	\end{enumerate}
	
\end{proposition}

\begin{lemma}[{\cite[Lemma 2.4]{ASposBull}}]
	\label{le-dis-supp}
	Let  $Y$ be  a uniformly monotone Banach  function space and $0 < \varepsilon < 1$. Assume that $f_1 $ and $f_2$ are positive elements in $Y$ such that
	$$
	\Vert f_1 + f_2 \Vert \le 1 \sem \text{\rm and} \sem  \frac{1}{ 1 + \delta ( \frac{\varepsilon}{3})}   \le 	\Vert f_1 - f_2 \Vert  ,
	$$
	where $\delta $ is the function satisfying  the definition of uniform monotonicity for $Y.$ 	Then there are two  positive functions   $h_1$ and $h_2$  in $Y$   with disjoint supports  satisfying that
	$$
	\Vert h_1 + h_2 \Vert  =1 	 \sem \text{\rm and} \sem \Vert h_i- f_i \Vert  <  \varepsilon \sep \text{for}\sep i=1,2  .
	$$
\end{lemma}

\begin{proposition}
	\label{comp}
	Let  $L$ be a  locally compact Hausdorff space and  $Y$  be  a Banach lattice.  Assume that   $S_1$ and $S_2$ are  positive operators in $L(C_0(L), Y)$ and $g_1$ and $g_2$ are  positive elements in $C_0(L)$.  Then the operator  $U: C_0(L) \llll Y$  defined by 
	$$
	U(f)= S_1 (f g_1)  +  S_2  (f g_2)  \seg (f \in C_0(L))
	$$
	satisfies  $\Vert U \Vert = \Vert S_1 ( g_1)  +  S_2  ( g_2) \Vert$.
\end{proposition}
\begin{proof}
		If $f \in B_{C_0(L)}$, we have 
	$$ 
	0 \le   \vert f \vert g_i \le  g_i, \sep i=1, 2.
	$$
	Since  	$S_i$ is positive for $i=1,2$,  we  have that 
	$$
	0 \le U( \vert f \vert) \le S_1 (g_1) + S_2 (g_2),
	$$ 
	so	
	$$
	 \Vert U(f)  \Vert \le  \Vert U(\vert f \vert)  \Vert  \le   \Vert S_1 (g_1) + S_2 (g_2)  \Vert.
	  $$
	As a consequence, 
	$$ 
	\Vert U  \Vert   \le   \Vert S_1 (g_1) + S_2 (g_2)  \Vert. 
	$$
	
	Given $\varepsilon > 0$, the subset $K$  given by
	$$
	K=\{ t \in L : g_1 (t) \ge \varepsilon \} \cup  \{ t \in L : g_2 (t) \ge \varepsilon \}
	$$
	is compact.  By Urysohn lemma  there is  a function $g \in  C_0(L)$ such that $0 \le g \le 1$ and such that $g(K)=\{1\}.$ So we have that
	$$
	\Vert g_i - g g_i \Vert _\infty \le \varepsilon, \sem  i=1, 2.
	$$
			Therefore 
		\begin{align*}
		\Vert U  \Vert  &  \ge   \Vert U(g)  \Vert   \\
		&=  \Vert S_1 (gg_1) + S_2 (gg_2)  \Vert   \\
		&\ge  \Vert S_1 (g_1) + S_2 (g_2)  \Vert  -  \varepsilon ( \Vert S_1 \Vert +  \Vert S_2 \Vert).  \\
		\end{align*}
By taking limit as $\varepsilon \to 0$ we obtain that $ \Vert U  \Vert \ge   \Vert S_1 (g_1) + S_2 (g_2)  \Vert $ and the  proof is finished.
\end{proof}

\begin{lemma}
\label{appro}	
Let $L$ be a  locally compact Hausdorff space and $Y$ be  a uniformly monotone Banach  lattice.  Let  $\delta $ be  the function satisfying  the definition of uniform monotonicity for the Banach lattice  $Y$ and    $ 0 < \eta < 1.$
Assume   that $f_0 \in S_{C_0(L)} $ and $S\in  S_{ L( C_0(L), Y)}	,$  $S \ge 0$ and  
$$
\Vert S(f_0) \Vert > \frac{1}{1+ \delta  (\eta ^2)} .
$$
Define  the sets $A_1, A_2, B_1 $ and $B_2$  by
$$
A_1=\Bigl\{ t \in L : -1 \le f_0(t) < -1 + \frac{\eta}{2} \Bigr \}, 
\sep  A_2=\Bigl\{ t \in L : -1 + \frac{\eta}{2} \le f_0(t) < -1 + \eta \Bigr\},
$$
$$
 B_1 =\Bigl\{ t \in L :  1- \frac{\eta}{2} <  f_0(t) \le 1 \Bigr\},
   \sep  B_2 =\Bigl\{ t \in L :  1- \eta <  f_0(t) \le 1- \frac{\eta}{2}\Bigr\}.
$$
There are   positive functions  $g_1$ and  $g_2 $ in $B_{C_0(L)}$  satisfying the following assertions
\begin{enumerate}
\item[a)] $   {g_1}_{\vert A_1} = 1$ and $ {g_1} _{\vert L \backslash (A_1 \cup A_2)} = 0.$
\item[b)] $   {g_2}_{\vert B_1} = 1$ and $ {g_2} _{\vert L \backslash (B_1 \cup B_2)} = 0.$
\item[c)]  $ \Vert g_1  +  f_0 g_1 \Vert _\infty \le \eta  $ and  so $\Vert   S(g_1) +S (  f_0 g_1 ) \Vert \le \eta .$
\item[d)]  $\Vert f_0 g_2 - g_2 \Vert _\infty \le \eta ,$ so $\Vert  S (  f_0 g_2) - S(g_2) \Vert \le \ \eta.$
\item[e)]   $\Vert S ( h) \Vert \le 2 \eta$ for every element  $h\in B_{C_0(L)}$ such that  $ h_{ \vert  A_1 \cup B_1}= 0 .$ As a consequence 
	\linebreak[4]
	$ \Vert S(fg_1 + fg_2-f) \Vert \le  2 \eta \Vert f \Vert _{\infty}  $  for every $f \in C_0(L).$ 
\end{enumerate}
\end{lemma}
\begin{proof}
	 We can define  the functions $g_1$ and $g_2$  as follows
	 $$
	 g_1(t) = 
	 \begin{cases} 
	 1 & \text{if } -1  \le  f_0(t) <   -1 + \frac{\eta}{2} \\
	 -\frac{2}{ \eta}  (f_0 (t) + 1 - \eta)     & \text{if }  -1 + \frac{\eta}{2}  \le  f_0(t) <  -1 + \eta\\
	 0       & \text{if }  -1 + \eta \le  f_0(t) \le  1 ,
	 \end{cases}
	 $$	 
	 $$
	 g_2(t) = 
	 \begin{cases} 
	 0       & \text{if }  -1 \le  f_0(t) \le  1- \eta \\
	 \frac{2}{ \eta}  (f_0 (t) - 1 + \eta)     & \text{if } 1- \eta <  f_0(t)  \le  1- \frac{\eta}{2}\\
	 1 & \text{if } 1- \frac{\eta}{2} <   f_0(t) \le  1.
	 \end{cases}
	 $$
It is immediate to check that $g_1$ and $g_2$ are positive functions in $B_{C_0(L)}$  satisfying  the conditions stated in a), b), c) and d). If   $h\in B_{C_0(L)}$  satisfies   $ h_{ \vert  A_1 \cup B_1}= 0 $
	we clearly have that  $\vert f_0\vert + \frac{\eta}{2} \vert h \vert \in S_{C_0(L)}$. 
	    By using that $S$ is a positive operator and the assumption  we have that
	\begin{align*}
	\Bigl\Vert S \Bigl( \vert f_0 \vert + \frac{\eta}{2} \vert h \vert \Bigr) \Bigr\Vert &  \le  1     \\
	&<  \Vert S (f_0) \Vert  \bigl( 1 + \delta (\eta ^2) \bigr)  \\
	&\le   \Vert S ( \vert f_0 \vert ) \Vert  \bigl( 1 + \delta (\eta ^2) \bigr) .
	\end{align*}
	
	By using the uniform monotonicity of $Y$ we obtain that 		$
	\Vert  S (\frac{\eta}{2} \vert h \vert) \Vert \le \eta ^2$ and so
	$$ \Vert  S ( h)  \Vert \le \Vert  S ( \vert h \vert)  \Vert \le 2 \eta.$$

	Finally notice that   $(g_1 + g_2  - 1)(A_1 \cup B_1)=\{0\} $  and $\Vert g_1 + g_2 - 1 \Vert _{\infty} \le  1$. So  from the previous part we conclude for every $f \in C_0(L), \sep \Vert S(fg_1 + fg_2-f) \Vert \le  2 \eta \Vert f \Vert _{\infty} .$
\end{proof}

	\newpage

\begin{theorem}
	\label{teo-BPBp-L-infty-UM-lattice}
	The pair $(C_0(L), Y) $ has the Bishop-Phelps--Bollob{\'a}s property for positive operators, for any locally compact Hausdorff topological space $L$, whenever  $Y$ is a uniformly monotone Banach function space. The function $\eta$  satisfying Definition \ref{def-BPBp-pos} depends only on the modulus of uniform monotonicity of $Y$.	
\end{theorem}
\begin{proof}
	Assume that $Y$ is a Banach function  space  on a measure space $(\Omega , \mu).$
	Let $ 0 < \varepsilon < 1$	 and $\delta $ be the function satisfying  the definition of uniform monotonicity for the Banach function  space $Y.$
	Choose a  real number  $\eta$ such that $0 < \eta = \eta (\varepsilon) < \frac{\varepsilon}{12}$  and  satisfying also
	$$
	\frac{1}{ 1 + \delta (\frac{\varepsilon}{ 18})} < \frac{1}{ 1 + \delta (\eta ^2)}  - 4 \eta.
$$

	Assume that $f_0 \in S_{C_0(L)} , S\in  S_{ L( C_0(L), Y)}	$ and  $S$ is a positive operator such  that 
	$$
	\Vert S(f_0) \Vert > \frac{1}{ 1 + \delta (\eta ^2)} .
	$$
Hence  we can apply Lemma \ref{appro} and so  there are positive functions $g_1$ and $g_2$ in $B_{ C_0(L)}$ satisfying all the conditions stated  in Lemma \ref{appro}, therefore 
	\begin{align}
	\label{SA-SB}
	\Vert S(g_1) - S(g_2 ) \Vert &  \ge  \Vert S( f_0 g_1) + S( f_0 g_2 ) \Vert  - \Vert   S(g_1)  + S(f_0 g_1) \Vert - \Vert S( f_0 g_2) -S(g_2 ) \Vert 
	\nonumber \\
	&\ge  \Vert S(f_0)\Vert - \Vert S( f_0 g_1+ f_0 g_2  - f_0) \Vert   - 2 \eta
	\\
	&\ge    \frac{1}{  1+ \delta (\eta ^2)}  - 4 \eta 
	\nonumber\\
	& >   \frac{1}{ 1+ \delta ( \frac{\varepsilon}{18})}.
	\nonumber
	\end{align}

	Since $S $ is a positive operator and   $\Vert S ( g_1) +  S(g_2) \Vert \le 1$, in view of \eqref{SA-SB} we can apply Lemma \ref{le-dis-supp}. Hence there are  two positive functions  $h_1$ and $h_2$ in $Y$  satisfying the following conditions
	\begin{equation}
	\label{h1-SA-h2-SB}
	\Vert h_1- S(g_1) \Vert  <   \frac{\varepsilon}{6},  \sem   \Vert h_2- S(g_2) \Vert   <    \frac{\varepsilon}{6} ,
	\end{equation}
	
	\begin{equation}
	\label{h-dis-sup-norm-one}
	\supp h_1 \cap  \supp h_2 = \varnothing  \sem \text{and} \sem \Vert h_1+h_2 \Vert  =1.
	\end{equation}
	
	Hence
	\begin{align}
	\label{S-A-out-sup-h1}	
	\Vert  S (g_1) \chi _{ \Omega  \backslash \supp h_1} \Vert  &  =  \Vert \bigl  (h_1 - S(g_1) \bigr)   \chi _{ \Omega  \backslash \supp h_1} \Vert 	  
	\nonumber  \\
	&\le     \Vert h_1 -  S (g_1)  \Vert  \\
	& <  \frac{ \varepsilon}{6} \sem \text{\rm (by \eqref{h1-SA-h2-SB})}
	\nonumber
	\end{align}
	and 
	\begin{align}
	\label{S-B-out-sup-h2}	
	\Vert  S (g_2) \chi _{ \Omega  \backslash \supp h_2} \Vert  &  =  \Vert   \bigl( h_2 - S(g_2) \bigr)   \chi _{ \Omega  \backslash \supp h_2} \Vert 	  
	\nonumber  \\
	&\le     \Vert h_2 -  S (g_2)  \Vert  \\
	& <  \frac{ \varepsilon}{6} \sem \text{\rm (by \eqref{h1-SA-h2-SB})}.
	\nonumber
	\end{align}

	Now we define the operator $U: C_0(L) \llll Y$ as follows
	$$
	U(f)= S(f g_1) \chi _{\supp h_1} +  S(f g_2) \chi _{\supp h_2} \seg (f \in C_0(L)).
	$$
	Since  $Y$ is a Banach  function space  and $S \in  L(C_0(L) ,Y),$ $U$  is well defined  and belongs to $ L( C_0(L) ,Y).$ The operator  $U$ is positive since $g_1$ and $g_2$ are positive elements in $C_0(L)$ and $S$ is a positive operator.  
	For any $f \in B_{ C_0(L)}$ we have that 
	\begin{align}
	\label{U-S}
	\Vert (U -  S)(f)  \Vert  & =   \Vert S(f g_1) \chi _{  \supp h_1} +  S(f g_2) \chi _{  \supp h_2}    - S(f) \Vert  
	\nonumber\\
	&  =    \Vert S(f g_1) \chi _{  \supp h_1} +S(f g_2) \chi_{  \supp h_2}   - S(f g_1 + f g_2 ) + S(f g_1 + f g_2  - f) \Vert \\
	&  \le  \Vert  S(g_1) \chi _{ \Omega  \backslash \supp h_1} \Vert  +  \Vert  S( g_2) \chi _{\Omega  \backslash  \supp h_2}   ) \Vert   +   \Vert S(f g_1 + f g_2  - f) \Vert \nonumber\\
	& < \frac{\varepsilon}{3} +   2 \eta  < \frac{\varepsilon}{2} \sem \text{(by   \eqref{S-A-out-sup-h1},  \eqref{S-B-out-sup-h2} and item e)  of Lemma  \eqref{appro})}.  
	\nonumber
	\end{align}
	Hence
	\begin{equation}
	\label{U-1}
	\vert \Vert U \Vert - 1 \vert < \dfrac{\varepsilon}{2},	
	\end{equation}
	so $U \ne 0$.	 
	
	Finally we define $T = \frac{U}{ \Vert U \Vert }$.   Since $U$ is a positive operator, $T$ is also positive. Of course
	\linebreak[4]
	$T \in S_{ L(C_0(L), Y)}$  and also satisfies
	\begin{align}
	\label{T-S}
	\Vert T - S \Vert & \le  \Vert T - U \Vert + \Vert U - S \Vert  
	\nonumber   \\
	& < \Bigl \Vert  \frac{U} { \Vert U \Vert } - U \Bigr \Vert + \frac{\varepsilon}{2} \sem \text{(by \eqref{U-S})} \\
	& = \bigl \vert 1 - \Vert U \Vert \bigr \vert + \frac{\varepsilon}{2} 
	\nonumber \\
	& < \varepsilon \sem \text{(by \eqref{U-1})}.
	\nonumber 
	\end{align}
	
	The  function $f_1$ given by 
	$$
	f_1(t) = 
	\begin{cases} 
	1 & \text{if } 1- \eta <  f_0(t) \le  1 \\
	\frac{f_0(t)}{1- \eta}       & \text{if } \vert f_0 (t)\vert \le  1- \eta\\
	-1       & \text{if } -1\le  f_0(t) <  -1+ \eta 
	\end{cases}
	$$
	 belongs  to $S_{C_0(L)}$ and satisfies that
	\begin{equation}
	\label{f1-f0}
	\Vert f_1 - f_0 \Vert _\infty  \le \eta  < \varepsilon.
	\end{equation}
	We clearly have that
	$$
	U(f_1) =   S(g_2) \chi _{ \sup h_2} -  S (g_1) \chi _{ \sup h_1} .
	$$
	Since $S\ge 0$  and  $g_1$ and $g_2$ are positive functions, in view of Proposition \ref{comp} we have that 
	$$
	\Vert U \Vert =  \Vert  S(g_1) \chi _{\supp h_1} +  S( g_2) \chi _{\supp h_2} \Vert.
	$$
	Since $h_1$ and $h_2$ have disjoint supports, for  each $x \in \Omega $ we obtain that
	\begin{align*}
	\bigl\vert \bigl ( S ( g_1) \chi _ {\sup h_1} +   S ( g_2) \chi _ {\sup h_2}   \bigr) (x) \bigr\vert 
  & = \bigl\vert \bigl (- S ( g_1) \chi _ {\sup h_1} +   S ( g_2) \chi _ {\sup h_2}   \bigr) (x) \bigr \vert     \\ 
	& =   \bigl\vert \bigl ( U( f_1 )  \bigr) (x) \bigr \vert .
	\end{align*}

	Since $Y$ is a Banach  function  space we conclude that 
	$$
	\Vert U \Vert =  \Vert  S(g_1) \chi _{\supp h_1} +  S( g_2) \chi _{\supp h_2} \Vert = \Vert  S(g_1) \chi _{\supp h_1} -  S( g_2) \chi _{\supp h_2} \Vert
	= \Vert U (f_1) \Vert.
	$$
	
	By \eqref{f1-f0} and    \eqref{T-S}, since $T$ attains its norm at $f_1$, the proof is finished
\end{proof}

\newpage

Now we will improve the statement in Theorem \ref{teo-BPBp-L-infty-UM-lattice}.  For that purpose  we  use operators in an  operator ideal, a notion that we recall below.

\begin{definition}
	[{\cite[Definition 9.1]{DF}}]
	\label{def-ideal}
	An  \textit{operator ideal $\mathcal{A}$} is a subclass of the class  $L$ of all continuous linear operators between Banach spaces such that for all Banach spaces $X$ and $Y$ its components 
		$$
		\mathcal{A}(X, Y) = L(X, Y) \cap \mathcal{A} 
		$$
		satisfy 

\noindent	
(1) $\mathcal{A}(X, Y)$ is a linear subspace of $L(X, Y)$ which contains the finite rank operators. 

\noindent		
(2) The ideal property: If $S \in\mathcal{A}(X_0, Y_0) ,$ $R \in L(X, X_0)$ and $T \in L(Y_0, Y),$ then the composition $TSR$ is in $ \mathcal{A}(X, Y).$
\end{definition}

We used in the previous proof that for a Banach function space $Y$ on a measure space $(\Omega, \mu)$, for any measurable set $A\subset \Omega,$ the operator $h\mapsto h \chi_A$ is linear and bounded.    So in case that $\mathcal{I} $ is some operator ideal  and the operator $S$  in the proof of Theorem \ref{teo-BPBp-L-infty-UM-lattice} belongs to  $\mathcal{I} (C_0(L),Y),$  then the operator $U$ also belongs to $\mathcal{I} (C_0(L),Y).$ Hence we  obtain the following result.

\begin{corollary}
Under the assumptions of Theorem \ref{teo-BPBp-L-infty-UM-lattice}, 	 if $\mathcal{I}$ is  some operator ideal, then the ideal $\mathcal{I}(C_0(L), Y)$ has the Bishop-Phelps-Bollobás property for positive operators.
\end{corollary}

Theorem 	\ref{teo-BPBp-L-infty-UM-lattice}
	is a far reaching extension of \cite[Theorems 2.5 and 3.2]{ASposBull}, where the same result was proved  in case that the domain space is a $L_{\infty} $ or $c_0$ space.
Our purpose now is to  obtain  a version of  Theorem  \ref{teo-BPBp-L-infty-UM-lattice}  for some abstract Banach lattices.
In order to get this result, we use that every uniformly monotone Banach lattice is order continuous (see \cite[Theorem 21, p. 371]{Bi}  and \cite[Proposition 1.a.8]{LiTz}).  Also any order continuous Banach lattice with a weak unit is order isometric to a Banach function space (see \cite[Theorem 1.b.14]{LiTz}). From Theorem \ref{teo-BPBp-L-infty-UM-lattice} and the previous argument we deduce the following result.

\begin{corollary}
	\label{cor-BPBp-pos-C0L-UM-lattice}
	The pair $(C_0(L), Y) $ has the Bishop-Phelps-Bollob{\'a}s property for positive
	\linebreak[4]
	operators, for any locally compact Hausdorff topological  space $L$ whenever  $Y$ is a uniformly monotone Banach lattice with a weak  unit. Moreover,  the function $\eta$  satisfying Definition \ref{def-BPBp-pos} depends only on the modulus of uniform monotonicity of $Y$.
\end{corollary}

 By using the previous result and the  same argument of   \cite[Corollary 3.3]{ASposBull} we obtain the next statement:

\begin{corollary}
	\label{cor-BPBp-pos-C0L-sep-UM-lattice}
 For any locally compact Hausdorff topological  space $L$ such that $C_0(L)$ is separable, The pair $(C_0(L), Y) $ has the Bishop-Phelps-Bollob{\'a}s property for positive operators whenever  $Y$ is a uniformly monotone Banach lattice. Moreover,  the function $\eta$  satisfying Definition \ref{def-BPBp-pos} depends only on the modulus of uniform monotonicity of $Y$.
\end{corollary}

Our intention now is to  provide a class of  Banach lattices  for which the previous result is in fact a characterization.   The proof of the next result  is a refinement of the  arguments used in \cite[Proposition 4.3]{ASposBull}, where it is assumed that  the domain has a non-trivial $M$-summand.

\begin{proposition}
	\label{prop-optimal}
	Let  $Y$ be a   strictly monotone Banach lattice and $L$ a locally compact Hausdorff topological space that contains at least two elements.  If the pair $(C_0(L),Y)$ has the BPBp for positive operators  then $Y$ is uniformly monotone.
\end{proposition}
\begin{proof}
It suffices to  show that $Y$ satisfies condition 2) in Proposition \ref{pro-char-UM} in case that the pair $(C_0(L),Y)$ has the BPBp for positive operators.
	
	Let   $ 0< \varepsilon < 1$. Let  us take elements  $u$ and $v$ in $Y$  such that 
	$$ 
	0 \leq u \leq v,  \sem   \|v\| = 1  \sem \text{and}  \sem \|v - u\| > 1 - \eta (\varepsilon),
	$$
	where $\eta$ is the function satisfying the definition of BPBp for positive operators for the pair $(C_0(L),Y)$.

	Since $L$ contains at least two elements there are $t_i \in L$ and  $f_i \in S_{C_0(L)}$    for $i=1,2$ such that 
	$$
	0\le f_i \le 1, \sem f_i(t_i)=1, \sep i=1,2,  \sem \text{and} \sem \supp f_1 \cap \supp f_2 = \varnothing.
	$$

	Define   the operator $S$ from $C_0(L)$  to $Y$ given by
	$$ 
	S(f) =   f(t_1)  (v-u) + f(t_2) u,  \sem   (f \in  C_0(L)).
	$$
	Since $v - u$ and $u$ are positive elements  in $Y,$     $S$ is a   positive operator from $C_0(L)$ to $Y$.
	
	By using that  $ v-u$ and $u$  are positive elements in $Y,$ for any element $f \in B_{C_0(L)}$  we have that
	\begin{eqnarray*}
		\label{norm-S}
		\nonumber
		\Vert S(f) \Vert & \le &  \Vert \;  \vert f(t_1) \vert \; ( v-u)  +  \vert f(t_2)  \vert \; u \Vert  \\
		& \le  &  \Vert v \Vert =1. 
	\end{eqnarray*}
	Since  we also have that $f_1+f_2 \in S_{C_0(L)}$ and  $S(f_1+f_2)= v \in S_Y,$ we obtain  that  $S \in S_{ L(C_0(L),Y)}.$
	Notice also that $\|S (f_1) \| = \|v-u\|  > 1 - \eta (\varepsilon) $.  By using that  the pair  $(C_0(L),Y)$ has the BPBp for positive operators and \cite[Remark 2.2]{ASposBull},   there exist a positive operator $T \in S_{L(C_0(L),Y)}$ and  a positive element $g \in S_{C_0(L)}$ satisfying 
	$$
	\Vert  T (g) \Vert =1, \sem    	
	\Vert T-S\Vert   < \varepsilon  \sem \text{and} \sem   \Vert  g-f_1  \Vert  < \varepsilon. 
	$$
	Hence, if $t \in L \backslash \supp f_1$ we have that $\vert g(t)\vert < \varepsilon.$

	By using that $T$ is a positive and normalized operator  we get that 
	$$ 
	1 = \Vert T(g ) \Vert \le   \Vert T(g + (1- \varepsilon ) f_2 )  \Vert \le 1.
	$$
	Since  $Y$ is strictly monotone we obtain that  $T(f_2)=0.$
	As a consequence, 
	$$
	\| u\|  = \| S( f_2)\| =  \| (S-T) (f_2)  \| \le \| S-T\| < \varepsilon.
	$$
	In view of  Proposition \ref{pro-char-UM}  we  proved that $Y$ is uniformly monotone. 
\end{proof}

Taking into account the previous result and Corollary \ref{cor-BPBp-pos-C0L-UM-lattice} we obtain the following  characterization.

\begin{corollary}
	\label{cor-optimal-char}
		Let  $Y$ be a   strictly monotone Banach lattice and $L$ a locally compact Hausdorff topological space that contains at least two elements.  If the pair $(C_0(L),Y)$ has the BPBp for positive operators  then $Y$ is uniformly monotone. In case that $Y$ has a weak unit the converse is also true.
\end{corollary}

\bibliographystyle{amsalpha}

\end{document}